\documentclass[a4paper, hidelinks, 11pt]{article}
\usepackage{graphicx} 
\usepackage{amsmath, relsize} 
\usepackage{todonotes} 
\usepackage{amsthm} 
\usepackage{amsfonts}
\usepackage{amssymb}
\usepackage{array}
\usepackage{url}
\usepackage{amsmath}
\usepackage{amssymb}
\usepackage{amsthm}
\usepackage{bbm}

\usepackage{enumerate}
\usepackage{setspace}
\usepackage{graphicx,color,hyperref}
\usepackage[margin=1in]{geometry}
\parskip \medskipamount
\newtheorem{theorem}{Theorem}[section]
\newtheorem{proposition}[theorem]{Proposition}

\newtheorem{lemma}[theorem]{Lemma}
\theoremstyle{definition}
\newtheorem{definition}[theorem]{Definition}
\newtheorem{remark}[theorem]{Remark}

\newcommand{\ind}[1]{{\text{\Large $\mathfrak 1$}}\left(#1\right)}

\newcommand{\lrc}[1]{\left\{#1\right\}}

\begin{document}

\title{Active phase for activated random walks on $\mathbb{Z}^d$, $ d \geq 3$, \\
with density less than one and arbitrary sleeping rate }
\author{Lorenzo Taggi\thanks{Technische Universit\"at Darmstadt, Darmstadt, DE; taggi@mathematik.tu-darmstadt.de. Supported by German Research Foundation (DFG).}}
\date{}
\maketitle

\begin{abstract}
 It has been conjectured that the critical density of  the Activated Random Walk model  is strictly less than one for any value of the sleeping rate.
 We prove this conjecture on $\mathbb{Z}^d$ when $d \geq 3$ and, more generally, on graphs where the random walk is transient. Moreover, we establish the occurrence of a phase transition on non-amenable graphs, 
 extending previous results which require that the graph is  amenable or a regular tree.
    MSC 2010 \emph{subject classifications.}
    Primary 82C22; 
    Secondary 60K35, 
              82C26. 
\end{abstract}

\section{Introduction}\label{sec:intro}
The activated random walk model (ARW) is a system
of interacting particles on a graph $G=(V,E)$.
Together with Abelian and Stochastic Sandpiles,
it belongs to a class of systems
which have been introduced 
in order to study a physical phenomenon known as 
self-organized criticality.
Moreover, it can be interpreted
as a toy model for an epidemic spreading, with
infected individuals moving diffusively on a graph.
The model is defined as follows.
Every particle is in one of two states, 
A (active) or S (inactive, sleeping).
Initially, the number of particles at each vertex of $G$ 
is an independent Poisson random variable with mean $\mu\in(0,\infty)$,  usually called the \emph{particle density}, and all particles are of type A. 
Active particles perform an independent, continuous time
random walk on $G$ with jump rate $1$, and
with each jump being to a uniformly random neighbour.
Moreover, every A-particle has a Poisson clock of rate $\lambda>0$
(called  \emph{sleeping rate}).
When the clock of a particle rings,
if the particle does not share
the vertex with other particles,
the particle becomes of type S; otherwise nothing happens.
Each S-particle does not move and remains sleeping until another particle jumps into its location.
At such an instant, the S-particle is activated and turns into type A.

For any value of $\lambda$, a phase transition as a  $\mu$ varies is expected to occur.
When $\mu$ is small, there is a lot of free space between the particles. This allows every particle
to turn into type S eventually and never become active again.
When this happens, we say that ARW \textit{fixates}.
This is not expected to occur when $\mu$ is large, since the active particles will repetitively jump
on top of other particles, activating the ones that had turned into type S.
In this case, we say that ARW is \textit{active}.

In a seminal paper \cite{Rolla}, Rolla and Sidoravicius prove a 0-1 law  (i.e., the process is either active or fixates with probability 1) and a monotonicity property with respect to $\mu$.
This leads to the existence of a critical curve $\mu_c = \mu_c(\lambda)$,
\begin{equation}
   \mu_c=\mu_c\left(\lambda\right)  := \inf\lrc{ \mu \geq 0 \, : \, \mathbb{P}(\text{ARW is active})  > 0   }.
   \label{eq:criticaldensity}
\end{equation}
which is such that, for any $\mu>\mu_c$ the system is almost surely active,
and for any $\mu<\mu_c$ the system fixates almost surely.
Though \cite{Rolla} is restricted to the case of $G$ being $\mathbb{Z}^d$,
the above properties hold for any vertex-transitive graph.
Throughout this paper we always consider that $G$ is an infinite simple graph that is locally-finite and vertex transitive, which ensures the existence of $\mu_c$.

In recent years considerable effort has been made to prove  basic properties of the critical curve $\mu_c = \mu_c(\lambda)$
\cite{Amir, Basu, Rolla, Rolla3, Rolla2, Sidoravicius, Stauffer, Taggi}.
A quite natural bound for this curve 
is  $\mu_c \leq 1$ for any value of $\lambda \in (0,\infty)$
, which was proved 
in  \cite{Amir, Rolla, Shellef}.
Indeed, one does not expect fixation when  the average number of particles per vertex is more than one,
since a particle can be in the S-state only if it is alone on a given vertex
and, for this reason, there is not enough space for all the A-particles to turn to the S-state.
A more challenging question is whether $\mu_c$ is strictly less than one  
 for any value of $\lambda \in (0, \infty)$, which is expected to hold true under wide generality.
 In other words, one expects that,
for any value of $\lambda \in (0,\infty)$, there exists a value of $\mu$ which is strictly less than
one such that, even though there is enough space for all the  particles to  turn into the S-state, 
particle motion prevents this from happening, so the system does not fixate.
This question was asked by Rolla and Sidoravicius in their seminal paper \cite{Rolla}
and appears also in  \cite{Basu, Dickman}.
Such a question received much attention in the last few years \cite{Rolla, Basu, Rolla2, Stauffer, Taggi} but,
despite much effort, a complete answer was provided only in two cases:
on vertex-transitive graphs where the random walk has a positive speed \cite{Stauffer}
and for a simplified model on $\mathbb{Z}^d$ where
the jump distribution of active particles is biased in a fixed direction \cite{Taggi}.
A partial answer which requires the assumption
that $\lambda$ is smaller than a finite constant $\lambda_0 < \infty$
 was also provided in
\cite{Stauffer} when $G$ is vertex-transitive and transient  
and in \cite{Basu} when $G = \mathbb{Z}$.

The first main result of this paper is the next theorem, which provides a positive
answer  to this question for any $\lambda \in (0, \infty)$ on $\mathbb{Z}^d$ , when $d \geq 3$, for the original model where active particles jump uniformly to nearest-neighbours. More generally, our result holds for any vertex-transitive amenable graph 
where the random walk is transient.
As a byproduct of our method, we 
also obtain that 
$\mu_c (\lambda) \rightarrow 0$  as $\lambda \rightarrow 0$
with a better convergence rate than as  in  \cite{Stauffer}.

\begin{theorem}
\label{theo1:transient graph}
If $G$ is vertex-transitive, amenable and  transient, then $$\mu_c(\lambda) < 1~~~~~ \mbox{$\forall \lambda \in (0, \infty)$.}$$ 
Moreover,
$\lim\limits_{\lambda \rightarrow 0} \frac{\mu_c(\lambda)}{\lambda^{\frac{1}{2}}} < \infty$. 
\end{theorem}

A second basic question concerning the behaviour of the critical curve $\mu_c=\mu_c(\lambda)$ is whether its value is  positive.
A positive answer has been proved by Sidoravicius and Teixeira in \cite{Sidoravicius} when $G= \mathbb{Z}^d$ by means of renormalization techniques. A shorter proof
was also provided  by  Stauffer and Taggi in \cite{Stauffer} when $G$ is amenable and vertex-transitive
and when $G$ is a regular tree.
The proofs of \cite{Sidoravicius, Stauffer} crucially rely on the amenability 
property of the graph or on the assumption that $G$ is a regular tree.

Our second main theorem provides a positive answer to this question 
on vertex-transitive graphs that are non-amenable, establishing the occurrence of
a phase transition for this class of  graphs and extending the previous results \cite{Sidoravicius, Stauffer}.
Moreover, we also obtain that $\lim_{\lambda \rightarrow \infty} \mu_c(\lambda)=1$.
\begin{theorem}
\label{theo: non amenable}
If $G$ is vertex-transitive and non-amenable, then $\mu_c(\lambda) > 0$ for any value of  $\lambda \in (0,\infty)$. More specifically,
$$
\mu_c(\lambda) \geq \frac{\lambda}{1 + \lambda} ~~~~~ \mbox{$\forall \lambda \in (0, \infty)$.}
$$
 \end{theorem}

 Theorem \ref{theo: non amenable} and the results of \cite{Sidoravicius, Stauffer}
imply that  $\mu_c(\lambda) >0$ for any $\lambda \in (0, \infty)$ 
 and that $\lim_{\lambda \rightarrow \infty} \mu_c(\lambda) = 1$
 on any vertex transitive graph.
Moreover, Theorem \ref{theo1:transient graph} and the results of \cite{Stauffer}
imply that
$\mu_c(\lambda) < 1$ for any $\lambda \in (0, \infty)$ 
and that $\lim_{\lambda \rightarrow 0} \mu_c(\lambda) = 0$
on any vertex-transitive  graph where the random walk is transient.

\paragraph{Description of the proofs}
Our proofs are simple and rely on 
a graphical representation,
which is called  \textit{Diaconis-Fulton} and has been introduced in \cite{Rolla}, 
and on  \textit{weak stabilization},
a procedure that has been introduced in \cite{Stauffer}
which consists of using the random instructions
of such a representation by following  a certain strategy.

A fundamental quantity for the mathematical analysis
of the activated random walks is  the number of times $m_{B_L}$
the origin is visited by a particle when the dynamics take place in a finite
ball of radius $L$, $B_L$,  with particles being absorbed whenever they leave $B_L$.
As it was proved in \cite{Rolla}, activity for ARW is equivalent
to the limit $L \rightarrow \infty$ of this quantity being infinite almost surely. 
A quantity that plays a central role in this paper
is the probability $Q(x,B_L)$ that an S-particle is at a vertex $x \in B_L$ when $B_L$ becomes stable.
This quantity is important since the values  $\{ \, Q(x,B_L) \, \}_{x \in B_L}$ are related 
to the expectation of $m_{B_L}$ by mass-conservation arguments.
Thus,  one can deduce whether the system is active by estimating these values.

The proof of Theorem \ref{theo1:transient graph} consists of bounding away from one
the probabilities $\{ Q(x,B_L)\}_{x \in B_L}$ for any  $\lambda \in (0, \infty)$ uniformly in $L$ and in $x \in B_L$.
This improves the upper bound that was provided in \cite{Stauffer},
where the probabilities  $\{ Q(x,B_L)\}_{x \in B_L}$ were bounded away from one
only for  $\lambda$ small enough.
Such an enhancement is obtained by introducing
a  stabilization procedure that allows to recover independence
from sleep instructions at one vertex. This gained independence and the fact that we
do not count the total number of instructions but only jump instructions, 
allows to obtain an additional factor
in the upper bound for $Q(x,B_L)$ which prevents this  bound
from exploding  when $\lambda$ is
infinitely large.
Our upper bound on $Q(x,B_L)$  implies that for any $\lambda \in (0, \infty)$
 one can find  $\epsilon>0$  and set the value of $\mu$ such that   $ 1 > \mu \geq   
Q(x,B_L ) + \epsilon$ for all $L$ and $x \in B_L$. 
This implies that a positive density $\epsilon$  of particles eventually leaves $B_L$
and, as it was proved in \cite{Rolla2}, that the system is active,
proving  Theorem \ref{theo1:transient graph}.

   Theorem \ref{theo: non amenable} extends to non-amenable graphs 
   the analogous result that was proved in \cite{Stauffer} for amenable graphs. 
   The idea of the proof that is presented in \cite{Stauffer} is that one assumes activity and uses this 
   assumption and the 
   weak stabilization procedure to show that  for any $\epsilon>0$,  
   there exists a large enough constant  $r_0=r_0(\epsilon)$ such that,
   for any large enough $L$ and for any vertex $x \in   B_L$
   which has a distance at least $r_0$ from the boundary of $B_L$,
     \begin{equation}\label{eq:intro}
   Q(x,B_L) \geq \frac{\lambda}{1+\lambda}- \epsilon.
   \end{equation}
   This leads to the conclusion that the particle density after the stabilization of $B_L$
   is at least $  \frac{\lambda}{1+\lambda}$.
   The amenability assumption is crucial here, since the number of particles which start `close' to 
   the boundary, for which (\ref{eq:intro}) does not hold,
    can be neglected  only if the graph is amenable (i.e. their number is of order $o(|B_L|)$).
   Since the initial particle density is $\mu$ and since the particle density cannot increase,
   we conclude that $\mu \geq \frac{\lambda}{1+\lambda}$.
   Since this is a consequence of  activity, we obtain that  $\mu_c \geq \frac{\lambda}{1+\lambda}$.


    In this paper, we use a different strategy that allows us to extend this result to non-amenable graphs.
    By assuming that the system is active and by using (\ref{eq:intro}),
    one obtains that the particle density in a small ball 
    $B_{(1-\delta)L} \subset B_L$ after the stabilization of the larger ball $B_L$ is at least
    $ \frac{\lambda}{1+\lambda}$, for some $\delta>0$ and all $L$ large enough.
    Thus, if we set $\mu < \frac{\lambda}{1+\lambda}$, this means that the particle density
    inside the smaller ball must have increased during the stabilization of the larger ball. Due to the conservation law,
    the only way  this might have happened
     is if a large number of particles which started from $B_L \setminus B_{(1-\delta) L}$
    turns into the S-state for the last time in  $B_{(1-\delta) L}$.
    We show that, if the graph is non-amenable,  this cannot happen simply because, even though the number of the boundary
    particles is not negligible if compared to $|B_{(1-\delta) L}|$, the  bias towards the outside of the ball allows only a
     few of them to penetrate inside the ball. So, the particle density in the smaller ball cannot increase
    and this leads to the conclusion that $\mu \geq \frac{\lambda}{1+\lambda}$. Since this is a consequence of activity, 
    we obtain that  $\mu_c \geq \frac{\lambda}{1+\lambda}$.

\vspace{0.2cm} 
The remaining part of the paper is organized as follows.
In Section \ref{sec:Diaconis} we introduce the Diaconis-Fulton representation following
 \cite{Rolla}, we recall the notion of weak stabilization 
  following \cite{Stauffer} and we fix the notation.
In Section \ref{sec:enforced stabilization} we provide an explicit upper bound for $Q(x,B_L)$, which is presented
in Theorem \ref{theo:boundsQ},
and we prove  Theorem \ref{theo1:transient graph}.
Finally, Section \ref{sec: Fixation on non-amenable graphs} is dedicated to the proof of Theorem \ref{theo: non amenable}.

\section{Diaconis-Fulton representation and weak stabilization}
\label{sec:Diaconis}
In this section we describe the
Diaconis-Fulton graphical representation
for the dynamics of ARW, following~\cite{Rolla},
and we recall the notion of \textit{weak stabilization}, following  \cite{Stauffer}.
Before starting, we fix the notation.

\textbf{Notation} A graph is denoted by $G=(V,E)$ and is always assumed to be simple, infinite and locally-finite.
The simple random walk measure is denoted by $P_x$, where $x$ is the starting vertex of the random walk.
The expectation with respect to $P_x$ is denoted by $E_x$.
For any set $Z \subset V$  and any pair of vertices $x, y \in V$, we let 
$$
G_{Z}(x,y) =  {E}_x \big ( \, \sum\limits_{t=0}^{\tau_Z-1} \mathbbm{1} \{  \, X(t) = y  \, \} \, \big)
$$
be the expected number of times a discrete time random walk $X(t)$ starting from $x$ hits $y$ before reaching $Z$ (Green's function), where $\tau_Z$ is the hitting time of the set $Z$.
If $Z = \emptyset$, then we set  $\tau_Z = \infty$ and we simply write $G(x,y)$.
We also denote by $\tau^+_Z$ the return time to $Z$.
The origin of the graph will be denoted by $0 \in V$.
We let $B_r = \{ y \in V \, \, : \, \, d(0,y) < r\}$ be the ball 
of radius $r>0$ centred at the origin, where $d( \cdot, \cdot )$ is the graph distance,
and we let $B_r(x) = B_r + x$ be the ball of radius $r$ which is centred at $x \in V$.

\paragraph{Diaconis-Fulton representation} 
For a graph $G=(V,E)$, the state of configurations is $\Omega=\{0,\rho,1,2,3,\ldots\}^V$, where a vertex being in state $\rho$ denotes that the vertex has 
one S-particle, while being in state $i\in\{0,1,2,\ldots\}$ denotes that the vertex contains $i$ A-particles.
We employ the following order on the states of a vertex: $0 < \rho < 1<2<\cdots$.
In a configuration $\eta\in \Omega$,
a vertex $x \in V$ is called \textit{stable} if
$\eta(x) \in \{0, \rho \}$,
and it is called \textit{unstable} if $\eta(x) \geq 1$.
We fix an array of  \textit{instructions} 
$\tau = ( \tau^{x,j}: \, x \in V, \, j \in \mathbb{N})$
(in this paper we assume that  $\mathbb{N}$ is the set of strictly positive integers),
where $\tau^{x,j}$ can either be of the form $\tau_{xy}$
or $\tau_{x\rho}$. We let $\tau_{xy}$ with $x,y\in V$ denote the instruction that a particle from $x$ jumps to vertex $y$, and $\tau_{x\rho}$ denote
the instruction that a particle from $x$ falls asleep.
Henceforth we call $\tau_{xy}$ a \emph{jump instruction} and $\tau_{x\rho}$ a \emph{sleep instruction}.
Therefore, given any configuration $\eta$, performing the instruction $\tau_{xy}$ in $\eta$ yields another configuration $\eta'$ such that 
$\eta'(z)=\eta(z)$ for all $z\in V\setminus\{x,y\}$, $\eta'(x)=\eta(x)-\ind{\eta(x)\geq 1}$, and $\eta'(y)=\eta(y)+\ind{\eta(x)\geq 1}$. We use the convention that $1+\rho=2$.
Similarly, performing the instruction $\tau_{x\rho}$ to $\eta$ yields a configuration $\eta'$ such that 
$\eta'(z)=\eta(z)$ for all $z\in V\setminus\{x\}$, and if $\eta(x)=1$ we have $\eta'(x)=\rho$, otherwise $\eta'(x)=\eta(x)$.

Let $h = ( h(x)\, : \,  x \in V)$ count the number of 
instructions used at each vertex.
We say that we \textit{use} an instruction
at $x$ (or that we \emph{topple} $x$) when we act on the current
particle configuration $\eta$ through the operator $\Phi_x$,
which is defined as,
\begin{equation}
\label{eq:Phioperator}
\Phi_x ( \eta, h) =
( \tau^{x, h(x) + 1}  \, \eta, \, h + \delta_x),
\end{equation}
where $\delta_x(y)=1$ if $y=x$ and $\delta_x(y)=0$  otherwise.
The operation $\Phi_x$ is \textit{legal} for $\eta$ if $x$ is unstable in $\eta$,
otherwise it is \textit{illegal}.

\vspace{0.8cm}

\noindent {\textbf{Properties}}
We now describe the properties of this representation.
Later  we discuss how they are related to the  stochastic dynamics of ARW.
For a sequence of vertices $\alpha = ( x_1, x_2, \ldots x_k)$,
we write $\Phi_{\alpha} = \Phi_{x_k} \Phi_{x_{k-1}}
\ldots \Phi_{x_1}$ and we say that $\Phi_{\alpha}$ is
\textit{legal} for $\eta$ if $\Phi_{x_\ell}$
is legal for $\Phi_{(x_{\ell-1}, \ldots, x_1)} (\eta,h) $
for all $\ell \in \{ 1, 2, \ldots k \}$.
Let $m_{\alpha} = ( m_{\alpha}(x) \, : \,x \in  V )$
be given by,
 $m_{\alpha}(x) \, = \, \sum_{\ell} \ind{x_\ell = x},$
the number of times the vertex $x$ appears in $\alpha$.
We write $m_{\alpha} \geq m_{\beta}$ if
$m_{\alpha} (x)  \,  \geq \, m_{\beta} (x) \, \, \, \forall x \in V$.
Analogously we write $\eta'   \geq   \eta$ if $\eta' (x) \, \geq \, \eta(x)$
for all $x \in V$. We also write $(\eta', h') \geq (\eta, h)$
if $\eta' \geq \eta$ and $h' = h$.

Let $\eta, \eta'$ be two configurations, $x$ be a vertex in $V$
and  $\tau$ be a realization
of the array of instructions. 
Let $V'$ be a finite subset of $V$. A configuration $\eta$ is said to be \textit{stable} in $V'$
if all the vertices $x \in V'$ are stable. We say that $\alpha$ is contained in $V'$
if all its elements are in $V'$, and we say that $\alpha$ \textit{stabilizes} $\eta$ in $V'$
if every $x \in V'$ is stable in $\Phi_\alpha \eta$.
The following lemmas give fundamental properties of the Diaconis-Fulton representation. 
For the proof, we refer to \cite{Rolla}. 

\begin{lemma}[Abelian Property]\label{prop:lemma2}
   Given any $V'\subset V$,
   if $\alpha$ and $\beta$ are both legal sequences for $\eta$
   that are contained in $V'$ and stabilize $\eta$ in $V'$, 
   then $m_{\alpha} = m_{\beta}$. In particular, $\Phi_{\alpha} \eta = \Phi_{\beta} \eta$.
\end{lemma}

For any subset $V'\subset V$, any $x\in V$, any particle configuration $\eta$, and any array of instructions $\tau$, we denote by $m_{V^{\prime},\eta,\tau}(x)$ the number of times that $x$ is toppled in the stabilization of $V'$ starting from configuration $\eta$ and using the instructions in $\tau$. Note that by Lemma~\ref{prop:lemma2}, we have that $m_{V^{\prime},\eta,\tau}$ is well defined.
\begin{lemma}[Monotonicity]\label{prop:lemma3}
   If $V' \subset V''\subset V$ and $\eta \leq \eta'$, then $m_{V', \eta, \tau} \leq m_{V'', \eta', \tau}$.
\end{lemma}

By monotonicity, given any growing sequence of subsets $V_1\subseteq V_2 \subseteq V_3\subseteq \cdots \subseteq V$ such that $\lim_{m\to\infty} V_m=V$, 
the limit 
$$
   m_{\eta, \tau} = \lim\limits_{m\to \infty} m_{V_m, \eta, \tau},
$$ 
exists and does not depend
on the particular sequence $\{V_m\}_m$.

We now introduce a probability measure on the space of instructions and of particle configurations.
We denote by $\mathcal{P}$ the probability measure according to which,
for any $x \in V$ and any $j \in \mathbb{N}$,
$\mathcal{P} (  \tau^{x,j} = \tau_{x\rho}   ) = \frac{\lambda}{1 + \lambda}$ and 
$\mathcal{P} (  \tau^{x,j} = \tau_{xy}   ) = \frac{1}{d(1 + \lambda)}$ for any $y\in V$ neighboring $x$,
where $d$ is the degree of each vertex of $G$ and the $\tau^{x,j}$ are independent across diffent values of $x$ or $j$.
Finally, we denote by $\mathcal{P}^\nu=\mathcal{P}\otimes \nu$ the joint law of
$\eta$ and $\tau$, where $\nu$ is a distribution on $\Omega$ giving the law of $\eta$.
Let $\mathbb{P}^\nu$ denotes the probability measure induced by the ARW process when the initial distribution of particles is given by $\nu$. 
We shall often omit the dependence on $\nu$ by writing $\mathcal{P}$ and $\mathbb{P}$ instead of $\mathcal{P}^\nu$ and $\mathbb{P}^\nu$.
The following lemma relates the dynamics of ARW to the stability property of the representation.
\begin{lemma}[0-1 law]
   \label{prop:lemma4}
   Let $\nu$ be a translation-invariant, ergodic distribution with finite density.
   Let $x\in V$ be any given vertex of $G$.
   Then $\mathbb{P}^{\nu}  (\text{ARW fixates} ) = \mathcal{P}^{\nu} ( m_{\eta, \tau} (x) < \infty ) \in \{0, 1 \}$.
\end{lemma}

Roughly speaking, the next lemma gives that removing an instruction sleep, cannot decrease the number of instructions 
used at a given vertex for stabilization.
In order to state the lemma, consider an additional instruction $\iota$ besides $\tau_{xy}$ and $\tau_{x\rho}$. The effect of $\iota$ is to  leave the configuration unchanged; i.e., $\iota \, \eta = \eta$,
so we will call this instruction  \textit{neutral}.
Then given two arrays $\tau = \left( \tau^{x,j} \right)_{x ,\, j }$ 
and $\tilde{\tau} = \left( \tilde{\tau}^{x,j} \right)_{x, \, j }$,
we write $\tau \leq \tilde{\tau}$ if for every $x \in V$ and $j \in \mathbb{N}$,
we either have $\tilde{\tau}^{x,j} = {\tau}^{x,j}$ or we have $\tilde{\tau}^{x,j} = \iota$ and 
${\tau}^{x,j} =  \tau_{x\rho}$.

\begin{lemma}[Monotonicity with enforced activation]
   \label{prop:lemma5}
   Let $\tau$ and $\tilde{\tau}$ be two arrays of instructions such that $\tau \leq \tilde{\tau}$.
   Then, for any finite subset $V' \subset V$ and configuration $\eta \in \Omega$,  we have
   $m_{V', \eta, \tau} \leq m_{V', \eta, \tilde{\tau}}.$
\end{lemma}

When we average over $\eta$ and $\tau$ using the measure $\mathcal{P}$, 
we will simply write $m_{V'}$ instead of $m_{V',\eta,\tau}$
and we will do the same for the other quantities that will be introduced later.

\subsection{Weak stabilization}

We now recall the notion of weak stabilization following \cite{Stauffer}.

\begin{definition}[weakly stable configurations]\label{def:wstable}
	We  say that a configuration $\eta$ is \emph{weakly stable} in a subset $K\subset V$ with respect to a vertex $x\in K$ 
	if $\eta(x)\leq 1$ and $\eta(y)\leq\rho$ for all $y\in K\setminus\{x\}$. 
	For conciseness, we just write that 
$\eta$ is weakly stable for $(x, K)$.
\end{definition}

\begin{definition}[weak stabilization]
   Given a subset $K\subset V$ and a vertex $x\in K$, the \emph{weak stabilization} of $(x,K)$ is a sequence of topplings of unstable vertices of $K\setminus\{x\}$ and of topplings of $x$ whenever $x$ has at least two active particles, until a weakly stable configuration for $(x,K)$ is obtained. The order of the topplings of a weak stabilization can be arbitrary.
\end{definition}

The Abelian property (Lemma \ref{prop:lemma2}),  the monotonicity property 
(Lemma \ref{prop:lemma3}) and  monotonicity with enforced activation 
(Lemma \ref{prop:lemma5}) hold for weak stabilization as well.
Since the proof of these lemmas is the same as for stabilization, for the proofs
we refer to \cite{Rolla}.
For any given particle configuration $\eta$ and instruction array $\tau$, we let 
$m^1_{(x,K), \eta, \tau}(y)$ be the number of instructions that are used at $y$
for the weak-stabilization of $(x,K)$.
By the Abelian property, this quantity is well defined.

We now formulate the Least Action Principle for weak stabilization of $(x, K)$.
In order to state the lemma, we need to extend the notion of unstable vertex and of legal operations 
to weak stabilization of $(x,K)$.
We call a vertex $y$ \textit{WS-unstable} (that is, unstable for weak stabilization)
in $\eta \in \Omega$ if $\eta(y) \geq 1 + \delta_x(y)$,
where $\delta_x(y)=1$ if $x=y$ and $\delta_x(y)=0$ otherwise.
We call a vertex $y$ \textit{WS-stable} in $\eta \in \Omega$ if
it is not WS-unstable.
We call the operation $\Phi_y$ defined in (\ref{eq:Phioperator})
\textit{WS-legal} for $\eta$ if $y$ is  WS-unstable in $\eta$.
Note that a WS-legal operation is always legal but a
legal operation is not necessarily WS-legal.
For a sequence of vertices $\alpha = ( x_1, x_2, \ldots x_k)$,
 we say that $\Phi_{\alpha}$ is
{WS-legal} for $\eta$ if $\Phi_{x_\ell}$
is WS-legal for $\Phi_{(x_{\ell-1}, \ldots, x_1)} (\eta,h) $
for all $\ell \in \{ 1, 2, \ldots k \}$.
We say that  that $\alpha$ \textit{stabilizes} $\eta$ weakly in  $(x,K)$
if every $x \in V$ is WS-stable in $\Phi_\alpha \eta$.
\begin{lemma}[Least Action Principle for weak stabilization of $(x,K)$]\label{prop:lemma1bis}
   If $\alpha$ and $\beta$ are sequences of topplings for $\eta$ such that
$\alpha$ is legal and stabilizes $\eta$ weakly in $(x,K)$ and $\beta$ is WS-legal and is contained in $K$,
then $m_{\beta} \leq m_{\alpha}$.
\end{lemma}
For the proof of the lemma, we refer to \cite{Stauffer}.

We now introduce  a stabilization procedure of $K$ consisting of a sequence of weak stabilizations of  $(x,K)$. This stabilization procedure is called \textit{stabilization via weak stabilization} and was used also in \cite{Stauffer}. From now on, we will omit the dependence of the quantities on $\eta$ and $\tau$, unless necessary, in order to lighten the notation.

\textbf{Stabilization via weak stabilization of $\boldsymbol{(x,K)}$.}
Let $\eta$ be the initial particle configuration.

\textit{First step.} We perform the weak stabilization of $(x,K)$. 
Recall that $m^{1}_{(x,K)}(y)$ is the total number of instructions that are used at $y$ for the weak stabilization of $(x,K)$ and let $\eta_1$ be the resulting particle configuration.
Note that, by definition of weak stabilisation, $\eta_1$ is either stable in $K$ or it is stable in $K \setminus \{x\}$ and it has one active particle at $x$. 
In the first case,  the stabilisation procedure is complete. 
In the second case 
we move to the second step.

\textit{$i$th step, for $i \geq 2$.}  We start by using  the next instruction at $x$
and we distinguish between two cases.

If this instruction
is sleep,  then we obtain a particle configuration which is stable in $K$.  In this case the stabilisation procedure is completed, we call $\eta_i$ the particle configuration we obtain and  define $m^i_{K, \eta, \tau}(y)$, the total  number of instructions which have been used at $y \in K$ up to this step, which by the Abelian property equals $m_{K, \eta, \tau}(y)$.

If this instruction is not sleep,  after using this instruction  we perform a new weak stabilisation of $(x,K)$. We call $\eta_i$ the particle configuration that we obtain and, for any $y \in K$, we let $m^{i}_{(x,K)}(y)$ be the number of instructions
that have been used at $y \in K$ up to this step. If $\eta_i$ is stable in $K$, then the procedure stops, otherwise we move to the next step and iterate.

Hence,  we iterate the procedure until we obtain a stable configuration. 
We let  $T_{(x,K)}$ denote the number of iterations,
\begin{equation}\label{eq:numberiterations}
T_{(x,K)} := \min\{n \in \mathbb{N}_{0} \, \, : \, \, \eta_n \mbox{ is stable } \},
\end{equation}
where $\eta_0 = \eta$ is the initial particle configuration.
Note that if $\eta$ is unstable in $K$ then  $T_{(x,K)}$ is strictly positive
and that, if $T_{(x,K)}=1$, then
the stable configuration $\eta_{T_{(x,K)}}$ hosts no particle at $x$.
For consistency, for any $i > T_{(x, K)}$,  we let $\eta_i$ be the stable configuration obtained after stabilizing $K$ and, for any $y \in K$, we define  $m^{i}_{(x, K)}(y)=m_K(y)$, which is the total number of instructions used
at $y$ for the complete stabilization of $K$. By the Abelian property, the quantities $T_{(x,K)}$ and $m^i_{(x,K)}$ are all well defined.
Note that the quantity $T_{(x,K)}$ is defined slightly differently than  in \cite{Stauffer}. Sometimes we will make explicit the dependence of 
$T_{(x,K)}$ on $\eta$ and $\tau$ by writing $T_{(x,K), \eta, \tau}$.
In Section \ref{sec:enforced stabilization} we will show that the number of weak stabilizations of $(x,K)$ that is necessary to perform to stabilize $K$ is related to the probability that the stabilization of $K$ ends with one particle at $x$,
which is an important quantity for the proof of Theorem \ref{theo1:transient graph}.
In Section \ref{sec:enforced stabilization} we will upper bound this 
probability by introducing a  new stabilization procedure which ignores the sleep instructions at $x$.

\section{Active phase on transient graphs}\label{sec:enforced stabilization}
In this section we prove Theorem \ref{theo1:transient graph}.
We first state Theorem \ref{theo:boundsQ}, where 
 the probability $Q(x,K)$ that the vertex $x \in K$ hosts an $S$-particle after  the stabilization of the finite set $K \subset V$ is  bounded away from one for any value of $\lambda \in (0, \infty)$. 
 In order for the next theorem to hold true the graph $G$ does not need to be vertex-transitive.

 \begin{theorem}
\label{theo:boundsQ}
Let  $G=(V, E)$ be a locally-finite graph and let $K \subset V$ be a finite set.
Then,  for any set $K \subset V$, for any vertex $x \in K$, 
for any positive integer $H$,
\begin{equation}   \label{eq:ubound}
   Q(x,K) \leq 1 - \Big( 1 -   \frac{G_{K^c}(x,x)}{H+1} \Big) 
    \Big ( \frac{1}{1 + \lambda} \Big)^H.
\end{equation}
\end{theorem}

Theorem \ref{theo1:transient graph}  is proved in the end of this section
and will be a direct consequence of Theorem  \ref{theo:boundsQ}.
We now introduce a new stabilization procedure that consists
of ignoring sleep instructions at one fixed vertex
and prove some auxiliary lemmas that are necessary for the proof of  
Theorem \ref{theo:boundsQ}. After that, we prove Theorem \ref{theo:boundsQ}
and Theorem \ref{theo1:transient graph}.

We  introduce the function $T^x$ that associates to any instruction
array $\tau$ a new instruction array $T^x(\tau)$ that is obtained from 
$\tau$ by ignoring all sleep instruction at $x$. 
More precisely, we define for any  $y \in V$ and $j \in \mathbb{N}$,
$$
\Big ( \, T^x(\tau)\, \Big )^{y,j} : =\begin{cases}
\tau^{y,j} &\mbox{ if $y\neq x$ }\\
\tau^{y,j}  &\mbox{ if $y = x$  and $\tau^{y,j} \neq \tau_{y\rho}$ } \\
\iota & \mbox{ if $y= x$  and $\tau^{y,j} = \tau_{y\rho}$},
\end{cases}
$$
recalling that  $\iota$ denotes a neutral instruction.
Moreover, for any $y, x\in V$, we let for any $i \in \mathbb{N}$,
\begin{equation}\label{eq:defenforced}
m^e_{(x,K), \eta, \tau}(y)   : = m_{K, \eta, T^x(\tau)}(y) 
\end{equation}
be the number of instructions that are used at $y$ when we stabilize the set 
$K$ by ignoring sleep instructions at $x$.
Moreover, recall the definition of stabilisation via weak stabilisation and define 
\begin{equation}
\label{eq:defenforcedstep}
m^{e,i}_{(x,K), \eta, \tau}(y)   : = m^i_{K, \eta, T^x(\tau)}(y).
\end{equation}
These functions plays an important role in this section.
For the proof of Theorem \ref{theo:boundsQ}, we will not  count the total number of 
instructions, but only the number of jump instructions. Thus, for any $y \in K$, we let
\begin{equation}\label{eq:defenforcedjump}
M^e_{(x,K), \eta, \tau}(y):  = \Big |  \, \Big \{ \tau^{y, j} \, \, : \, \,  \, j \in [0, m^e_{(x,K),  \eta, \tau}(y)],  \, \, \,  \tau^{y, j}  \neq \tau_{y\rho} \Big \} \, \Big |
\end{equation}
be the number of jump instructions that are used at $y$ when we stabilize $K$ by ignoring
sleep instructions at $x$.
Similarly, we let $M_{K, \eta, \tau}(y)$ be the  number of jump
instructions that are used at $y$ for the  stabilization of  $K$
and  $M^1_{(x,K), \eta, \tau}(y)$ be the number of jump instructions
that are used at $y$ for the weak stabilization of $(x,K)$.
In the next lemma we state some simple but important relations between these quantities.
Recall the definition of the variable $T_{(x,K)}$,   (\ref{eq:numberiterations}), 
which counts the number 
of weak stabilisations of $(x,K)$ which is necessary to perform in order 
to stabilise $K$.

\begin{lemma}\label{lemma:independence}
Let $\eta$ be an arbitrary particle configuration, 
let $\tau$ be an arbitrary instruction array,
suppose that $T_{ (x,K), \eta, \tau } < \infty$,
let $\tilde \tau = T^x(\tau)$ be obtained from $\tau$ by turning all the sleep instructions
at $x$ into a neutral instruction.
Then,  for any vertex $y \in K$,
for any $ i \in \{1, \ldots, T_{ (x,K), \eta, \tau } -1  \}$, 
\begin{align}\label{eq:independ1}
m^i_{(x,K), \eta, \tau} (y) \, &= \, m^i_{(x,K), \eta,  \tilde \tau} (y),
~~~~~
M^i_{(x,K), \eta, \tau} (y) \, = \, M^i_{(x,K), \eta,  \tilde \tau} (y), \\
\label{eq:independ4}
m^e_{(x,K), \eta,  \tau} (y) \, & \geq \, m_{K, \eta,   \tau} (y),
~~~~~~~~~
M^e_{(x,K), \eta, \tau} (y) \, \geq \, M_{K, \eta,   \tau} (y), \\
\label{eq:independ5}
&   
~~~~~~~~~T_{(x,K), \eta,  \tilde \tau} \,  \geq  T_{(x,K), \eta, \tau}.
\end{align}
Moreover, if the particle configuration which is obtained by stabilising $\eta$ in $K$ has no particle at $x$, we deduce that (\ref{eq:independ1}) holds also for $i = T_{(x,K), \eta, \tau}$.
 \end{lemma}
\begin{proof}
Recall the definition of stabilisation via weak stabilisation.  We perform a stabilisation via weak stabilisation and we check that 
at every step (\ref{eq:independ1}) holds.
\textbf{Step $\boldsymbol{i=1}$. }The first step consists in the weak stabilisation of $(x,K)$.
By definition of weak stabilisation,  when we perform the  weak stabilization of $(x,K)$, we topple $x$ only if $x$ contains 
at least two particles, so the sleep instructions at $x$ have no effect.  Hence,  we deduce that,
\begin{equation}\label{eq:check1}
m^{1}_{(x,K), \eta, \tau}(y) =  m^1_{(x,K), \eta,  \tilde \tau},
 \quad M^{1}_{(x,K), \eta, \tau}(y) =  M^1_{(x,K), \eta,  \tilde \tau}.
 \end{equation}
If  the configuration we obtain, $\eta_1$,  is not stable, we go to the next step.
If the configuration we obtain is stable, this means that  the particle configuration which is obtained by stabilising $\eta$ in $K$ has no particle at $x$
and that $T_{(x,K)} = 1$,
hence the proof is concluded in this case.

\textbf{Step $\boldsymbol{i \geq 2}$. }
The $i$th step starts by using the next instruction at $x$ and,  if this instruction is not slee in performing a weak stabilisation of $(x,K)$ afterwards. We denote by $\eta_i$
the particle configuration that we obtain at the end of the $i$th step.
Note that in the previous steps we checked that,
\begin{equation}\label{eq:checkingsteps}
\forall j = 1, 2, \ldots, i-1, \quad m^{j}_{(x,K), \eta, \tau}(y) =  m^{j}_{(x,K), \eta,  \tilde \tau},
 \quad M^{j}_{(x,K), \eta, \tau}(y) =  M^{j}_{(x,K), \eta,  \tilde \tau},
\end{equation}
We distinguish between three cases.

 \textbf{Case (i):} 
 The first case is that the first instruction we use at $x$ 
is a sleep instruction.  In this case the particle configuration we obtain, $\eta_i$,  is stable,  it hosts one sleeping particle at $x$, and $T_{(x,K), \eta, \tau} = i$,  hence (\ref{eq:independ1}) is fulfilled for any $i=1, 2, \ldots, T_{(x,K), \eta, \tau}-1$
since we checked (\ref{eq:checkingsteps}) in the previous steps.

\textbf{Case (ii):}  The second case is that the first instruction we use at $x$  is not a sleep instruction and that the weak stabilisation we perform afterwards ends with a stable configuration in $K$, i.e, $\eta_i$ is stable in $K$ and $T_{(x,K)} = i$.  By definition of weak stabilisation of $(x,K)$, this can only happen if no particle jumps from a neighbour of $x$ to $x$ while performing the  weak stabilisation,  hence $\eta_i(x) = 0$.
Hence,  in this case no sleep instruction 
is used at $x$ during the $i$th weak stabilisation and for this reason and for the fact that in the previous steps we checked (\ref{eq:checkingsteps}) we  deduce that
(\ref{eq:independ1}) holds for any $i=1, 2, \ldots, T_{(x,K), \eta, \tau}$.

\textbf{Case (iii):} The third case is that the first instruction we use at $x$ is not a sleep instruction and that  the weak stabilisation we perform afterwards ends with a particle configuration which is unstable  in $K$. 
Observe that this necessarily means (by definition of weak stabilisation) that the configuration we obtain, $\eta_i$, is stable in $K \setminus \{x\}$ and that it hosts an active particle at $x$.
Since by definition of weak stabilisation  $x$ is toppled only if it contains at least $2$ active particles at $x$,  then the sleep instructions used at $x$ have no effect. From this and from the fact that at the previous steps we checked that  (\ref{eq:checkingsteps}) holds,
we deduce that,
$$
m^{i}_{(x,K), \eta, \tau}(y) =  m^{i}_{(x,K), \eta,  \tilde \tau},
 \quad M^{i}_{(x,K), \eta, \tau}(y) =  M^{i}_{(x,K), \eta,  \tilde \tau}.
$$
We now move to the  step $i+1$ and iterate.

We iterate the  procedure until the last step, $i = T_{(x,K)}$,  which is the first step such that Case (i) or (ii) are fulfilled.  Hence,  we checked that  (\ref{eq:checkingsteps}) holds up to the last step $i = T_{(x,K)}$ and that, if the procedure ends with Case (iii), then 
(\ref{eq:independ1}) holds also for $i = T_{(x,K)}$.  This proves  (\ref{eq:independ1}) 
for any $i = 1, \ldots,T_{(x,K)}-1$ and also proves the last  claim in the statement of the lemma.  

The relations (\ref{eq:independ4}) follow from  a direct application of monotonicity with enforced activation for stabilization  (Lemma \ref{prop:lemma5}).

For (\ref{eq:independ5}),  we compare the stabilisation via weak stabilisation procedure 
for $\tau$ and $\tilde \tau$ simultaneously.
First of all note that,  while stabilising via weak stabilisation using the instructions of $\tau$,  all the sleep instructions which have been used at $x$ during the first $T_{(x,K), \eta, \tau}-1$ steps  had no effect, hence ignoring them makes no difference up to this step. 
Now consider the last step for the stabilisation via weak stabilisation which uses the instructions of $\tau$.  If such  last step starts with a sleep instruction at $x$ (as described  in Case (i)), then  the particle configuration gets stabilised, i.e,  $\eta_{  T_{(x,K), \eta, \tau}}$ is stable in $K$. This however is not true for
$\eta_{  T_{(x,K), \eta, \tilde \tau}}$, since the array $\tilde \tau$ has neutral instruction at $x$ in place of sleep instructions,  hence we deduce that the stabilisation via weak stabilisation
which uses the instructions of $\tilde \tau$ may perform further steps, i.e, 
$T_{(x,K), \eta, \tilde \tau}  \geq  T_{(x,K), \eta,  \tau}$.
Instead,  if such last step starts with a jump instruction at $x$ (as described in Case (ii)),
this means that no sleep instruction of $\tau$ was ever used at $x$ during such last step,  hence ignoring sleep instructions at $x$ makes no difference when we compare the stabilisation-via-weak stabilisation with $\tau$ and $\tilde \tau$ and for this reason 
$
T_{(x,K), \eta, \tau} = T_{(x,K), \eta, \tilde \tau},
$
and that
$\eta_{  T_{(x,K), \eta, \tau}} = \eta_{  T_{(x,K), \eta, \tilde \tau}}$
in this case.
Combining the two cases we deduce  (\ref{eq:independ5}) and conclude the proof.

\end{proof}

For the next lemma we need to recall the notion
of stabilization via weak stabilization that has been introduced in Section \ref{sec:Diaconis}, recall also (\ref{eq:defenforcedjump}).
Define,
\begin{equation}\label{eq:Afunction}
A_{(x,K), \eta, \tau} : = M^e_{(x,K), \eta,  \tau} (x)  - M^1_{(x,K), \eta,  \tau} (x).
\end{equation}
be the total number of jump instructions that are used at $x$  when
we stabilize $K$ by ignoring sleep instructions at $x$ and that are not used for the  weak stabilization of  $(x,K)$.

\begin{lemma}\label{lemma:Qandenforced}
Let $G=(V,E)$ be an arbitrary locally-finite graph and let $K \subset V$ be a finite set. 
Let $\eta^{\prime}$ be the particle configuration that is obtained after the stabilization of $K$.
Then, for any $x \in K$, for any integer $\ell \geq 2$,
\begin{equation}\label{eq:mainupperbound}
\mathcal{P} \big(  \eta^{\prime}(x) = \rho, \, \, T_{(x,K)} = \ell \big) \leq \frac{\lambda}{1+\lambda} \, \, \big( \frac{1}{1+\lambda} \big)^{\ell-2} \, \, \mathcal{P} \big(   A_{(x,K)} \geq \ell-2   \big)
\end{equation}
\end{lemma}
\begin{proof}
Recall the stabilization-via-weak-stabilization procedure that has been introduced in Section \ref{sec:Diaconis} and  the definitions (\ref{eq:defenforced}),  
(\ref{eq:defenforcedstep}), 
(\ref{eq:defenforcedjump}).
First of all, note that for any integer $\ell \geq 2$, 
\begin{align*}
\mathcal{P} \big( \, \eta^{\prime}(x)  = \rho, \, \, T_{(x,K)} = \ell \,  \big ) & =
\mathcal{P} \big(  \, T_{(x,K)} \geq \ell, \,  \,   \tau^{x, m_{(x,K)}^{\ell-1}(x)+1} = \tau_{x \rho}  \, \big ) \\
& = \frac{\lambda}{1+\lambda} \, \, \mathcal{P} \big(  \, T_{(x,K)} \geq \ell \,  \big ).
\end{align*}
In the previous display 
$\tau^{x, m_{(x,K)}^{\ell-1}(x)+1}$
is the first instruction which is used at $x$  during the $\ell$th step
of the stabilisation via weak stabilisation of $(x,K)$.
This first equality holds true since, by definition of stabilisation via weak stabilisation, the event $\{  \eta^{\prime}(x) = \rho, T_{(x,K)}= \ell  \}$
occurs  if and only if
$T_{(x,K)} > \ell -1$ 
(i.e., the $\ell-1$th  step does not end with a stable configuration)
and the first instruction
used at $x$ during the  $\ell$-th step is sleep.
The second equality above follows from the independence of the instructions. Now note that,
\begin{align*}
\Big \{  T_{(x,K)} \geq \ell \Big \}   
& =   \Big \{ \forall i \in [1, \ell-2], \, \,  \,  \, \, \tau^{x, m^i_{(x,K)}(x)+1} \, \neq \tau_{x\rho} \mbox{ and } 
m^{i+1}_{(x,K)}(x) > m^{i}_{(x,K)}(x)   \Big \} \cap 
\Big \{  T_{(x,K)} \geq \ell \Big \}   \\
& = 
\Big \{ \forall i \in [1, \ell-2], \, \,  \,  \, \, \tau^{x, m^{i,e}_{(x,K)}(x)+1} \, \neq \tau_{x\rho} \mbox{ and } 
m^{i+1,e}_{(x,K)}(x) > m^{i,e}_{(x,K)}(x)   \Big \} \cap 
\Big \{  T_{(x,K)} \geq \ell \Big \}  \\
& \subset 
\Big \{ \forall i \in [1, \ell-2], \, \,  \,  \, \, \tau^{x, m^{i,e}_{(x,K)}(x)+1} \, \neq \tau_{x\rho} \mbox{ and } 
m^{i+1,e}_{(x,K)}(x) > m^{i,e}_{(x,K)}(x)   \Big \} \cap 
\Big \{  T^e_{(x,K)} \geq \ell \Big \}.
\end{align*}
The first identity holds true since, in order for the stabilization-via-weak-stabilization procedure to consist of at least $\ell \geq 2$  steps, it is necessary that 
the first instruction used at $x$ during the steps  $j=2, 3, \ldots,$ $\ell-1$ 
is not a sleep instruction, but a jump instruction.
The second identity and the  inclusion follow from Lemma \ref{lemma:independence}.
From the fact that the function $T^e_{(x,K)}$ is independent from the
sleep instructions at $x$ and from the previous inclusion relation we deduce that,
\begin{align}
\begin{split}\label{eq:split}
& \mathcal{P} \Big (  T_{(x,K)} \geq \ell   \Big )  \\
& \leq 
\mathcal{P} \Big ( \big  \{ \forall i \in [1, \ell-2], \, \,  \,  \, \, \tau^{x, m^{i,e}_{(x,K)}(x)+1} \, \neq \tau_{x\rho} \mbox{ and } 
m^{i+1,e}_{(x,K)}(x) > m^{i,e}_{(x,K)}(x)\big  \}   \cap \big  \{  T^e_{(x,K)} \geq \ell \big \} \Big ) \\
& \leq   \big ( \frac{1}{1 + \lambda } \big )^{\ell - 2} \, \, 
\mathcal{P} \Big ( \big \{ \forall i \in [1, \ell-2], \,  
m^{i+1,e}_{(x,K)}(x) > m^{i,e}_{(x,K)}(x)\big  \}   \cap \{  T^e_{(x,K)} \geq \ell  \big \} \Big ) \\
& =   \big ( \frac{1}{1 + \lambda } \big )^{\ell - 2} \, \,  
\mathcal{P} \big (
T^e_{(x,K)} \geq \ell \big ) \\
& \leq  \big ( \frac{1}{1 + \lambda } \big )^{\ell - 2}  \mathcal{P} \Big ( M^e_{(x,K)}(x) - M^1_{(x,K)}(x) \geq \ell -2  \Big ),
\end{split}
\end{align}
For the last inequality we used the fact that,  if the stabilization via weak stabilization of $(x,K)$   consists of at least $\ell \geq 2$  steps, 
then it is necessarily the case that at least $\ell - 2$ jump instructions are used at $x$
after the first step.
This concludes the proof.
\end{proof}

\begin{remark}
In \cite{Stauffer} the quantity in the left-hand side of (\ref{eq:mainupperbound}) is bounded from above by 
the probability that at least $\ell-2$  instructions
 are used at $x$ after the first weak stabilization, without distinguishing between jump and sleep instructions.
Our enhancement is obtained by counting only the jump instructions which are used for the stabilisation and by introducing 
a stabilization procedure, (\ref{eq:defenforced}), that ignores sleep instructions at one vertex. This allows us to recover independence from sleep instructions and, thus, to split  the upper bound in (\ref{eq:split}) into the product of two  factors, which are then bounded from above separately.
\end{remark}

 In the next lemma we will bound from above the expectation of $A_{(x,K)}$.

\begin{lemma}\label{lemma:upperboundA}
Let $G$ be a locally-finite graph and let $K \subset V$ be a finite set. Then, for any $x \in K$,
\begin{equation}\label{eq:upperboundA}
\boldsymbol{E} \Big ( A_{(x,K)} \Big ) \leq G_{K^c}(x, x),
\end{equation}
where $\boldsymbol{E}$ is the expectation with respect to $\mathcal{P}$.
\end{lemma}
\begin{proof}
Note that the expectation of $A_{(x,K)}$ can be written as follows,
\begin{equation}\label{eq:claim0}
\boldsymbol{E} \Big ( A_{(x,K)} \Big ) = \sum\limits_{k=0}^{\infty} \, Poi_{\mu}(k) \,  \, \Big [ 
  \boldsymbol{E}_k \big ( M^e_{(x,K)}(x) \big )  \, -  \,  \boldsymbol{E}_k \big ( M^1_{(x,K)}(x)\big ) \, \,  \Big ]
\end{equation}
where $\boldsymbol{E}_k$  is the expectation $\boldsymbol{E}$ conditional on having precisely $k$ particles starting from $x$ at time $0$ and $Poi_{\mu}(k)$ is the probability that a Poisson random variable with mean $\mu$ has outcome $k$.
We claim that, for any $k \in \mathbb{N}$,
\begin{equation}
\label{eq:claim1}
\boldsymbol{E}_{k+1} \Big ( M^1_{(x,K)}(x) \Big) = \boldsymbol{E}_k \Big (M^e_{(x,K)}(x) \Big ),
\end{equation}
and that
\begin{equation}
\label{eq:claim3}
\boldsymbol{E}_{k+1} \Big ( M^1_{(x,K)}(x) \Big ) \leq \boldsymbol{E}_{k} \Big ( M^1_{(x,K)}(x) \Big ) \, + \, G_{K^c}(x,x).
\end{equation}
By using (\ref{eq:claim1}) and (\ref{eq:claim3}) we obtain from (\ref{eq:claim0}) that,
\begin{align*}
\boldsymbol{E} \Big ( A_{(x,K)} \Big ) & = 
\sum\limits_{k=0}^{\infty} \, Poi_{\mu}(k) \,  \, \Big [ 
  \boldsymbol{E}_{k+1} \big ( M^1_{(x,K)}(x) \big )  \, -  \,  \boldsymbol{E}_k \big ( M^1_{(x,K)}(x)\big ) \, \,  \Big ]  \\
 &  \leq  \sum\limits_{k=0}^{\infty} \, Poi_{\mu}(k)\, \,  G_{K^c}(x, x) \\ 
 & = G_{K^c}(x, x),
\end{align*}
obtaining the desired inequality (\ref{eq:upperboundA}).
So, in order to conclude the proof, it remains to prove (\ref{eq:claim1}) and (\ref{eq:claim3}).

The equality (\ref{eq:claim1}) holds true since adding one particle at a $x$ and never moving that particle
is equivalent to stabilizing $K$ 
by ignoring all the sleep instructions at $x$.
For a formal proof,  let  $\eta^{k+1}$ be an arbitrary particle configuration with $k+1$ particles at $x$
and let $\eta^k$ be  obtained from $\eta^{k+1}$ by removing
one of the particles at $x$.
Let $\tau$ be an arbitrary array and let $\tilde \tau$ be obtained from $\tau$ by turning
sleep instructions at $x$ into a neutral instruction. We use  the instructions of $\tau$  for $\eta^{k+1}$ and 
the instructions of $\tilde \tau$ for $\eta^k$ simultaneously.
More specifically,  let $\alpha = (x_1, x_2, \ldots x_{|\alpha|})$ be a sequence 
that stabilizes $\eta^k$ in $K$ by using the instructions of $\tilde \tau$.
Since any step of $\alpha$ is legal 
for $\eta^{k}$ when we use $\tilde \tau$ (a neutral instruction is always legal), it is also WS-legal for  $\eta^{k+1}$
when we use $\tau$.
Moreover, since $\Phi_{\alpha} \eta^k$ is stable
in $K$ and has no particle at $x$ when we use $\tilde \tau$, then 
$\Phi_{\alpha} \eta^{k+1}$ is weakly stable in $(x,K)$
when we use  $ \tau$.
Thus, the sequence $\alpha$ stabilizes $\eta^k$ in $K$
when we use the instrutions of $\tilde \tau$ and stabilizes 
$\eta^{k+1}$ weakly in $(x,K)$ when we use the instructions of $\tau$.
From the Abelian property  we deduce that for any $y \in K$,
$
m^1_{(x,K), \eta^{k+1}, \tau}(y) = m^e_{(x,K), \eta^{k}, \tau}(y).
$
This implies (\ref{eq:claim1}).

We now prove  (\ref{eq:claim3}), adapting the steps of a similar proof that appears in \cite{Stauffer}
to our  setting.
Let $\eta$ be an arbitrary particle configuration with $k+1$ particles at $x$.
In the first step, we move one of the particles that is at $x$ until it leaves the set $K$, ignoring any sleep instruction.
During  this step of the procedure we might use some instruction at $x$
that is WS-illegal (but legal).
The  expected number of times a jump instruction is used at  $x$ during this step 
 is  $G_{K^c}(x,x)$.
In the second step, we perform weak the stabilization of $(x,K)$ with the remaining particles.
Let $M_{K, \eta, \tau}^{\prime}(x)$ be the total number of jump instructions that are used at $x$. We have that,
by monotonicity with enforced activation and by the least action principle
for weak stabilization,
$$
 M^1_{(x,K), \eta, \tau}(x)  \, \,  \leq  \, \, M_{K, \eta, \tau}^{\prime}(x) .
$$
Moreover, since in the second step we start from a configuration with $k$ particles at $x$ and instructions are independent,
$$
\boldsymbol{E}_{k+1} \big ( \,  M_K^{\prime} (x)  \, \big ) =  \boldsymbol{E}_k \big ( \,  M^1_{(x,K)}(x) \, \big ) + G_{K^c}(x,x) .
$$
By using the two previous relations we obtain (\ref{eq:claim3}).

\end{proof}

\subsection{Proof of Theorem \ref{theo:boundsQ}}
The proof of Theorem \ref{theo:boundsQ} is a direct consequence of 
Lemma \ref{lemma:Qandenforced} and  Lemma \ref{lemma:upperboundA}.

From Lemma \ref{lemma:Qandenforced} we obtain that,
\begin{align*}
Q(x,K) & = 
\sum\limits_{ \ell=2}^{\infty} \, \mathcal{P} \big(  \eta^{\prime}(x) = \rho, \, \, T_{(x,K)} = \ell \big)  \\
& \leq \sum\limits_{ \ell=2}^{\infty}  \frac{\lambda}{1+\lambda} \, \, \big( \frac{1}{1+\lambda} \big)^{\ell-2}  
   \, \,  \mathcal{P} \big(   A_{(x,K)} \geq \ell -2  \big),
\end{align*}
having used the fact that  $\mathcal{P} \big(  \eta^{\prime}(x) = \rho, \, \, T_{(x,K)} = 1 \big) =0$
and that $T_{(x,K)} > 0 $ almost surely.
 We now perform simple calculations in order to prove the quantitative upper bound of Theorem \ref{theo:boundsQ}.
By using the Markov's inequality and Lemma \ref{lemma:upperboundA}, we obtain that for any positive integer $H$,
\begin{align*}
\mathcal{P} \big(  \eta^{\prime}(x) = \rho \big) &  \leq
 \sum\limits_{ \ell=0}^{H-1}  \frac{\lambda}{1+\lambda} \, \, \big( \frac{1}{1+\lambda} \big)^{\ell}  
\, +\, \sum\limits_{ \ell=H}^{\infty}  \frac{\lambda}{1+\lambda} \, \, \big( \frac{1}{1+\lambda} \big)^{\ell}  
    \mathcal{P} \big(   A_{(x,K)} \geq \ell   \big) \\
    & \leq \frac{\lambda}{1+\lambda} \, \, \Big [  \sum\limits_{ \ell=0}^{H-1}  \, \big( \frac{1}{1+\lambda} \big)^{\ell}  
\, +\, \, \frac{G_{K^c}(x,x)}{H+1} \,    \sum\limits_{ \ell=H}^{\infty}   \, \, \, \big( \frac{1}{1+\lambda} \big)^{\ell}   \Big ] = \\
    & \leq \frac{\lambda}{1+\lambda} \, \, \Big [ \frac{1}{1 - \frac{1}{1 + \lambda}}  \, \, - \, \,    \Big( 1 -   \frac{G_{K^c}(x,x)}{H+1} \Big) 
    \Big ( \frac{1}{1 + \lambda} \Big)^H \, \, \frac{1}{ \Big ( 1 - \frac{1}{1 + \lambda} \Big)}\, \, \Big ] \\
    & = 1 - \Big( 1 -   \frac{G_{K^c}(x,x)}{H+1} \Big) 
    \Big ( \frac{1}{1 + \lambda} \Big)^H.
\end{align*}
This concludes the proof of Theorem \ref{theo:boundsQ}.

\subsection{Proof of Theorem \ref{theo1:transient graph}}
Suppose  that the graph is vertex-transitive and transient.
We have that for any set $K \subset V$ and any vertex $x \in K$, 
$G_{K^c}(x,x) \leq G(0,0) < \infty$. 
Thus, if we replace $G_{K^c}(x,x)$ by $G(0,0)$ in the right-hand side of (\ref{eq:ubound}),
the inequality is still true.
Moreover, if we set $H$ large enough, we have that the right-hand side of  (\ref{eq:ubound})
is bounded away from one uniformly in  $K$ and in $x \in K$ for any $\lambda >0$.
We can then find a function $g(\lambda)$ such that,
\begin{equation}\label{eq:bound2}
 \forall K \subset V, \, \, \, \,  \, \, \, \, \forall x \in K, \, \, \, \,  \, \, \, \,Q(x,K) \leq g(\lambda)<1 .
\end{equation}
Moreover, by choosing  $H^{*} :=    \lceil \sqrt{  \frac{G(0,0)}{\log(1 + \lambda)}  } \rceil$,
we deduce that $g(\lambda)$ can be chosen such that 
$\lim_{\lambda \rightarrow 0} \frac{g(\lambda)}{\lambda^{\frac{1}{2}}} < \infty$.

Suppose now that $\mu > g(\lambda)$.
Since from (\ref{eq:bound2}) we have that the expected number of particles after the stabilization of $K$ is at most 
$g(\lambda) \,  |K|$, it follows that the expected number of particles leaving $K$ during the stabilization
of $K$ is at least $ (  \, \mu - g(\lambda) \, )  \, \, | K|$.
Since a positive density of particles leaves the set,
since  the graph is amenable, and since $K$ is an arbitrary
finite set, we deduce from  
\cite{Rolla2}[Proposition 2] that
the system is active. This implies that
 $\mu_c(\lambda) \leq g(\lambda)$
 for any $\lambda \in (0, \infty)$ and concludes the proof.

\section{Fixation on non-amenable graphs}
 \label{sec: Fixation on non-amenable graphs} 
In this section we prove Theorem \ref{theo: non amenable}.
We start with an auxiliary lemma, which provides an
upper bound for the expected number of times that the particles
which start from the vertices that are `close' to the boundary of a ball
visit the  centre of that ball.
Afterwards, we use this lemma to prove Proposition \ref{prop:fixation non-amemable},
showing that the probabilities  $\{Q(x,B_{L}) \}_{x \in B_L}$
must fulfil a certain condition.
Finally, we prove Theorem \ref{theo: non amenable} 
by showing that, if one assumes that ARW is active
and that $\mu < \frac{\lambda}{1 + \lambda}$, then 
such a condition is violated, 
obtaining a contradiction.

\begin{lemma}\label{lemma:RWboundary}
Let $G$ be a vertex-transitive graph where the random walk has a positive speed. 
There exists $C_1=C_1(G) < \infty$ such that, for any $\delta \in (0,1)$,
there exists an infinite increasing sequence of integers  $\{L_n\}_{n \in \mathbb{N}}$
such that
$$
\sum\limits_{x \in B_{L_n} \setminus B_{(1 - \delta) \, L_n} }  G_{B_{L_n}^c} \big (x, 0 \big ) \,  \leq \, C_1 \, \delta  \,  L_n.
$$
\end{lemma}
\begin{proof}
For any pair of real numbers $r_2 > r_1$, let $$\Xi(r_1,r_2) : = 
E_0 \Big ( \sum\limits_{t=0}^{\infty}   \mathbbm{1}\{ X(t) \in B_{r_2}\setminus B_{r_1} \} \Big),$$
be the expected number of vertices in the ring $B_{r_2} \setminus B_{r_1}$
which are visited by the random walk,
with $\Xi(0,r_2) $ being the expected  number of vertices in the ball  $B_{r_2}$
which are visited by the random walk.
By regularity of the graph, for any integer $n$ and $x \in B_{n}$, we have that,
\begin{equation}\label{eq:regularity}
P_{x}\big ( \tau_0 < \tau_{B_n^c} \big ) \, = G_{B_n^c \cup \{0\}}(x,x) \, \, P_0 \big ( \tau_x < \tau_{ \{0\} \cup B_n^c}^+ \big).
\end{equation}
Then, for any $\delta^{\prime} \in (0,1)$,
\begin{align}
\sum\limits_{x \in B_n \setminus B_{(1 - \delta^{\prime}) \, n}} G_{B_n^c}(x,0) \, \, & = \, \, G_{B_n^c}(0,0) \sum\limits_{x \in B_n \setminus B_{(1 - \delta^{\prime}) \, n}} P_{x}\big ( \tau_0 < \tau_{B_n^c} \big )  \\
\label{eq:condtris}
& \leq G(0,0)^2 \, \, \sum\limits_{x \in B_n \setminus B_{(1 - \delta^{\prime}) \, n}} P_{0}\big ( \tau_x < \tau_{B_n^c} \big ) \\
\label{eq:cond4}
& \leq G(0,0)^2 \, \, \Xi \Big ( \, (1 - \delta^{\prime}) n , \, n \, \Big ),
\end{align}
where we used (\ref{eq:regularity})  and vertex-transitivity.
We have that, 
\begin{equation}\label{eq:cond1}
\forall n \in \mathbb{N} ~~~~~ \Xi \big (  0, n \big )  \geq  \sum\limits_{\ell=1}^{  \lfloor \frac{1}{\delta^{\prime}} \rfloor } \, \Xi \big (  \, \,  \delta^{\prime} \, n  \, (\ell-1),\,   \delta^{\prime} \, n   \, \ell \, \,  \big ) .
\end{equation}
Since the random walk has a positive speed, we have that there exists $K=K(G)$ such that, 
\begin{equation}\label{eq:cond2}
\forall n \in \mathbb{N} ~~~~~  \Xi \big (  0, n \big ) \, \leq \, K \, \, n,
\end{equation}
(see for example \cite{Stauffer}[eq. (5.16)] for a proof).
Conditions (\ref{eq:cond1}) and (\ref{eq:cond2}) imply that,
\begin{equation}\label{eq:cond3}
\forall n \in \mathbb{N} ~~~~~~ \exists \ell_n \in [\frac{1}{2 \delta^{\prime}} , \frac{1}{\delta^{\prime}}] ~~~~~~ \mbox{s.t.} ~~~~~~\Xi \big (  \, \delta^{\prime} \,  n \,  (\ell_n-1) ,\,  \delta^{\prime} \, n  \, \ell_n  \,   \big )  \leq 4 \, K \, \delta^{\prime} \, n.
\end{equation}
For any $n \in \mathbb{N}$, define now  $L_n :=\lfloor  \delta^{\prime} \, n \, \ell_n \rfloor$. From (\ref{eq:cond3}) we obtain that, for any large enough $n$,
\begin{multline}
\Xi \big (  \, L_n ( 1 \, - \, \frac{\delta^{\prime}}{2} )  , \, \,  L_n \,   \big ) \,  \leq
\Xi \big (\, L_n ( \frac{\delta^{\prime} n \ell_n}{L_n} \, -  \delta^{\prime}), L_n    \big ) \leq 
\Xi \big (\, L_n ( \frac{\delta^{\prime} n \ell_n}{L_n} \, -  \frac{1}{\ell_n}) , L_n    \big )
= \\
\Xi \big (\, \delta^{\prime} n \ell_n\, -  \frac{L_n}{\ell_n} , L_n    \big ) \leq 
\Xi \big (\,\delta^{\prime} n  (\ell_n\, -  1), \delta^{\prime} n \ell_n    \big ) 
\leq  \, 4 \,  K \,  \, \frac{\delta^{\prime} \, n \, \ell_n}{\ell_n} \leq
5 K \frac{L_n}{\ell_n} \leq 10 K  \delta^{\prime} L_n.
\end{multline}
The proof follows by  defining $\delta =  \frac{\delta^{\prime}}{2}$,
$C_1 = 20 \, K \, G(0,0)^2$ and by selecting an infinite increasing subsequence
of $\{L_n\}_{n \in \mathbb{N}}$.
\end{proof}

\begin{proposition}\label{prop:fixation non-amemable}
Let $G$ be vertex-transitive and suppose that the random walk on $G$ has a positive speed.
Then, for any  $\delta \in (0,1)$, there exists an  infinite  increasing sequence of integers $\{L_n\}_{n \in \mathbb{N}}$ such that 
\begin{equation}\label{eq:constraintQ}
\sum\limits_{x \in B_{(1 - \delta) \, L_n}} \, \, \,  G_{B_{L_n}^c} (x,0) \, 
\, \, \big (  Q(x,B_{L_n})  \,  - \,   \mu  \big ) \, \, \, \leq  \mu \, \, C_1  \, \, \delta\,  L_n,
\end{equation}
where $C_1=C_1(G)$ is the constant that has been defined in Lemma \ref{lemma:RWboundary}.
\end{proposition}
\begin{proof}
The expected number of particles visiting the origin is related to the 
quantities $\{ Q(y,B_L)\}_{y \in B_L}$ by the following relation,
\begin{equation}\label{eq:relationQm}
\mathbb{E}_{B_L}\big ( M_{B_L}(0) \big) = \sum\limits_{y \in B_L} \,  G_{B^c_L}(y, 0)  \, \, \big ( \,   \mu - Q(y, B_L) \,  \big),
\end{equation}
where $M_{B_L}(y)$ is the total number of jump instructions which are used at $y \in B_L$ for the stabilization of $B_L$.
We will first prove (\ref{eq:relationQm}) and then use it to prove the proposition.
In order to prove (\ref{eq:relationQm}), we use the ghost explorer technique similarly
to \cite{Shellef, Stauffer}.
First, we let the particles move until the ball $B_L$ is stable.
This means that some particles leave  $B_L$ being absorbed at the boundary
and other particles remain in $B_L$ after having turned into the S-state.

We now let a \textit{ghost particle} start an independent 
simple random walk from every vertex that is occupied by an S-particle in $B_L$.
Ghost particles are `killed' whenever they visit  $B_L^c$.
We let $R_{B_L}(x)$ be the total number of visits at  $x \in B_L$ by ghost particles. We let 
$W_{B_L}(x)$ be the total number of visits at $x$ by normal particles or by ghost particles.
Now we have that,
$$
M_{B_L}(x) = W_{B_L}(x) - R_{B_L}(x).
$$
As both particles and ghost particles stop only when they leave $B_L$ and as random walks
are independent, we have that,
$$
 \tilde{E}_{B_L}[ W_{B_L}(x)]  =  \mu \,  \sum\limits_{y \in B_L} \,  G_{B^c_L}(y, x),
$$
where $\tilde {E}_{B_L}$ is  the expectation in the enlarged probability space
of activated random walks and ghost particles.
Moreover, since precisely one ghost leaves from every vertex where an S-particle
is located, 
$$
 \tilde{E}_{B_L}[ R_{B_L}(x)]  =   \sum\limits_{y \in B_L} \,  G_{B^c_L}(y, x) \, \, Q(y, B_L),
$$
by linearity of expectation.
The proof of (\ref{eq:relationQm}) is concluded again by using linearity of expectation.

By using (\ref{eq:relationQm}) and
Lemma \ref{lemma:RWboundary}, we obtain that, for any $\delta \in (0,1)$, there exists an
infinite increasing sequence of integers $\{L_n\}_{n \in \mathbb{N}}$ such that,
$$
\mathbb{E}\big (\, M_{B_{L_n}}(0) \, \big ) \leq \,  \,  \sum\limits_{y \in B_{(1 - \delta) L_n }} \, \, \Big ( \,  \mu - Q(y,B_L) \,  \Big )  \, \, G_{B_{L_n}^c}(y,0) \, \, \, \, + \, \, \,\mu \,  C_1 \,  \delta \,  L_n.
$$
Since the left-hand side of the previous inequality has to be non-negative for any $n$, we obtain (\ref{eq:constraintQ}),
concluding the proof.

\end{proof}

\begin{proof}[\textbf{Proof of Theorem \ref{theo: non amenable}}]
We will show that, if $G$ is such that the random
walk has a positive speed, then  condition (\ref{eq:constraintQ}) cannot be  satisfied
for an infinite increasing sequence $\{L_n\}_{n \in \mathbb{N}}$
when $\mu < \frac{\lambda}{1+\lambda}$ and ARW is active,
obtaining a contradiction and concluding that ARW fixates
when $\mu < \frac{\lambda}{1+\lambda}$.

First of all, note that 
\begin{equation}\label{eq:lowerboundQ}
Q(x, B_L)  \geq \, \mathcal{P} \big ( \, m_{B_{ L} (x)}(x) \geq 1 \,  \big ) \, \frac{\lambda}{1+\lambda} .
\end{equation}
This inequality was proved in \cite{Stauffer} and follows from the next relation,
\begin{equation}\label{eq:relation1}
\{   m^1_{(x, B_{ L} (x))}(x) \geq 1\}  \, \, \cap \, \, \{ \tau^{x, m_{(x,K)}^1(x) + 1} = \tau_{x \rho}  \} \, \,  \subset \, \,  \{ \eta^{\prime}(x) = \rho    \},
\end{equation}
Indeed, if one concludes weak stabilization of $(x,K)$ with one particle at $x$ and the next instruction at $x$ is sleep,
then the stabilization is completed with one particle at $x$.
Moreover, at least one instruction is used at $x$ during the stabilization of $K$ if and only if
at least one instruction is used at  $x$ during the weak stabilization of $(x,K)$.
By independence of instructions, one obtains
(\ref{eq:lowerboundQ}).

Thus, assume that $\mu < \frac{\lambda}{1+\lambda}$ and that ARW is active
and let $D : =  \frac{\frac{\lambda}{1+\lambda} - \mu}{2} >0$. 
We have that, for any $\delta>0$ and for any $L$ large enough depending on $\delta$,
\begin{align}
\forall x \in B_{(1-\delta) L }, ~~~~~~~
Q(x, B_L) \,
& \geq \, \mathcal{P} \big ( \, m_{B_{\delta L} (x)}(x) \geq 1 \,  \big ) \, \frac{\lambda}{1+\lambda}  
\label{eq:step1bis} \\
 & \geq \, \mathcal{P} \big ( \, m_{B_{\delta L}}(0) \geq 1 \,  \big ) \, \frac{\lambda}{1+\lambda}  
\label{eq:step1tris} \\
& \geq \mu + D.
\label{eq:step2}
\end{align}
where the first inequality follows from (\ref{eq:lowerboundQ})
and  from  monotonicity (Lemma \ref{prop:lemma3}), for the second  inequality we used vertex-transitivity and 
for the third inequality we used the definition of activity and Lemma \ref{prop:lemma4}.

For any $\delta \in (0,1)$ and for any $L$ large enough the next inequality holds,
\begin{align}
\sum\limits_{y \in B_{ (1-\delta) \, L }} G_{B_{L}^c}(y,0)\Big ( Q(y,B_L ) \, - \, \mu \Big )  & 
  \geq  D \sum\limits_{y \in B_L} G_{B^c_L}(y,0)  \\
& =  \, \, D \, G_{B_{L  }^c}(0,0) \, \,  \mathbb{E}_0 \Big (
\sum\limits_{t=0}^{\tau_{B^c_L}-1}  \mathbbm{1} \big\{ X(t) \in B_{(1-\delta)L} \big\} \Big ) \\
\label{eq:violation}
& \geq   \, \, D  \, G_{B_{L}^c}(0,0) \, \, p \, \,  (1-\delta) \, L ,
\end{align}
where for the first inequality we used (\ref{eq:regularity}) and 
for the second inequality we used conditional expectation and
we let  $p = P_0 \big ( X(t) \neq 0 \, \, \forall t >0 \big ) >0$ be the probability that the random walk 
does not return to its starting vertex, which is positive since the random walk has a positive speed
and is then transient.
Choose now $\delta \in (0,1)$ small enough such that, for any $L$ large enough,
\begin{align*}
\sum\limits_{y \in B_{ (1-\delta) \, L }} G_{B_{L}^c}(y,0)\Big ( Q(y,B_L ) \, - \, \mu \Big ) 
& \geq   \, \, D  \, G_{B_{L}^c}(0,0) \, \, p \, \,  (1-\delta) \, L \\
& > C_1\,  \delta \,  L.
\end{align*}
Since the previous condition holds for any $L$ large enough, we deduce that
an infinite increasing sequence $\{L_n\}_{n \in \mathbb{N}}$ satisfying  (\ref{eq:constraintQ}) 
cannot exist when $\mu \leq \frac{\lambda}{1+\lambda}\leq1$, obtaining the desired contradiction.
\end{proof}

\section*{Acknowledgements}
The author thanks Elisabetta Candellero and Alexandre Stauffer for very interesting discussions
and Shirshendu Ganguly and Leonardo Rolla for  useful comments.


\begin{thebibliography}{10}

\bibitem{Amir} 
G. Amir, O. Gurel-Gurevich:\textit{
On Fixation of Activated Random Walks.}
{Electron. Comm. Probab.},
Vo. {15}, No. 12, (2009).

%

\bibitem{Basu} 
R. Basu, S. Ganguly, C. Hoffman:
\textit{Non-fixation of symmetric Activated Random Walk on the line for small sleep rate.}
{ArXiv: 1508.05677} (2015).
Accepted on Comm. Math. Phys.

 

\bibitem{Dickman} 
R. Dickman, L.T. Rolla, V. Sidoravicius:
\textit{Activated Random Walkers: Facts, Conjectures and Challenges.}
{J. Stat. Phys.},
 Vo. {138}, 126-142, (2010).




\bibitem{Rolla} 
L. T. Rolla and V. Sidoravicius:
\textit{Absorbing-State Phase Transition for Driven-Dissipative
Stochastic Dynamics on $\mathbb{Z}$.}
{Invent.  Math.},
Vo. {188},  No. 1,  pp. 127-150,  (2012).


\bibitem{Rolla3}
L. T. Rolla, V. Sidoravicius, O. Zindy:
\textit{Critical Density for Activated Random Walks}.
arXiv: 1707.06081 (2017).

\bibitem{Rolla2} 
L. T. Rolla and L.Tournier:
\textit{Sustained Activity for Biased Activated Random Walks at Arbitrarily Low Density.}
{ArXiv: 1507.04732} (2015). 
Accepted on Ann. Inst. Henri Poincar\'e.

\bibitem{Shellef} 
E. Shellef:
\textit{ Nonfixation for activated random walk.
}{ALEA},
Vo 7, pp. 137-149, (2010).

\bibitem{Sidoravicius}
V. Sidoravicius and  A. Teixeira:
\textit{Absorbing-state transitions for Stochastic Sandpiles and 
Activated Random Walk.}
Electr.  J. Probab.,
Vo. 22 (2017),  No. 33.


\bibitem{Stauffer} 
A. Stauffer and L. Taggi:
\textit{Critical density of activated random walks on transitive graphs}.
ArXiv: 512.02397 (2015).
Accepted on Ann. Probab.

\bibitem{Taggi} 
L. Taggi:
\textit{Absorbing-state phase transition in biased activated random walk.} 
Electr. J.  Probab.,
Vo. 21  (2016),  No. 13.




\end{thebibliography}
\end{document}